\documentclass[11pt,letterpaper]{article}
\usepackage[normalem]{ulem}
\usepackage{soul}
\usepackage{amsmath}
\usepackage{amsfonts}
\usepackage{amssymb}         
\usepackage{amsopn}
\usepackage{theorem}          
\usepackage{latexsym}
\usepackage{graphicx}        
\usepackage{mathrsfs}		
\usepackage{psfrag}
\usepackage{algorithmic}
\usepackage{algorithm}
\usepackage{color}
\usepackage{verbatim}
\usepackage{url}

\newtheorem{result}{\ }[section]
\theoremstyle{changebreak}                
\newtheorem{thm}[result]{Theorem}

\newtheorem{lem}[result]{Lemma}

\newtheorem{prop}[result]{Proposition}

\newenvironment{proof}
 {{\sl Proof.}\hspace*{1 ex}}%
 {{\nopagebreak\hspace*{\fill}$\Box$\par\vspace{12pt}}}

\definecolor{plgreen}{rgb}{0,0.5,0}

\begin{document}

\thispagestyle{empty}
\begin{center} 

{\LARGE Gaussian random projections for Euclidean membership problems}
\par \bigskip
{\sc Vu Khac Ky\footnote{Supported by a Microsft Research Ph.D.~fellowship.}, Pierre-Louis Poirion, Leo Liberti} 
\par \bigskip
{\small {\it CNRS LIX, \'Ecole Polytechnique, F-91128 Palaiseau,
France} \\ Email:\url{vu,poirion,liberti}@lix.polytechnique.fr} 
\par \medskip \today
\end{center}
\par \bigskip

\begin{abstract}
We discuss the application of random projections to the fundamental
problem of deciding whether a given point in a Euclidean space belongs
to a given set. We show that, under a number of different assumptions,
the feasibility and infeasibility of this problem are preserved with
high probability when the problem data is projected to a lower
dimensional space. Our results are applicable to any algorithmic
setting which needs to solve Euclidean membership problems in a
high-dimensional space.
\end{abstract}


\section{Introduction}
Random projections are very useful dimension reduction techniques
which are widely used in computer science \cite{AppEmbedding,
  LectureEmbedding}. We assume we have an algorithm $\mathcal{A}$
acting on a data set $X$ consisting of $n$ vectors in $\mathbb{R}^m$,
where $m$ is large, and assume that the complexity of $\mathcal{A}$
depends on $m$ and $n$ in a way that makes it impossible to run
$\mathcal{A}$ sufficiently fast. A random projection exploits the
statistical properties of some random distribution to construct a
mapping which embeds $X$ into a lower dimensional space $\mathbb{R}^k$
(for some appropriately chosen $k$) while preserving distances,
angles, or other quantities used by $\mathcal{A}$.

One striking example of random projections is the famous
Johnson-Lindenstrauss lemma \cite{JLL}:
\begin{thm}[Johnson-Lindenstrauss Lemma]
  \label{t:jll}
Let $X$ be a set of $m$ points in $\mathbb{R}^m$ and
$\varepsilon>0$. Then there is a map $F:\mathbb{R}^m\to\mathbb{R}^k$
where $k$ is $O(\frac{\log m}{\varepsilon^2})$, such that for any
$x,y\in X$, we have
\begin{equation}
(1 - \varepsilon) \|x-y\|^2_2 \le \|F(x)-F(y)\|^2_2 \le (1 +
  \varepsilon) \|x-y\|^2_2.\label{eq:jll}
\end{equation}
\end{thm}
Intuitively, this lemma claims that $X$ can be projected in a much
lower dimensional space whilst keeping Euclidean distances
approximately the same. The main idea to prove Thm.~\ref{t:jll} is to
construct a random linear mapping $T$ (called {\it JL random mapping}
onwards), sampled from certain distribution families, so that for each
$x\in\mathbb{R}^m$, the event that
\begin{equation}
  (1-\varepsilon)\|x\|^2_2 \le\|T(x)\|^2_2\le (1+\varepsilon)\|x\|^2_2
  \label{eq:rp}
\end{equation}
occurs with high probability. By Eq.~\eqref{eq:rp} and the union
bound, it is possible to show the existence of a map $F$ with the
stated properties (see \cite{DataFriendly, ElementaryJLL}).

In this paper we employ random projections to study the following
general problem:
\begin{quote}
  {\sc Euclidean Set Membership Problem} (ESMP). Given $p\in
  \mathbb{R}^{m}$ and $X\subseteq\mathbb{R}^m$, decide whether $p \in
  X$.
\end{quote}
This is a fundamental class consisting of many problems, both in {\bf
  P} (e.g.~the {\sc Linear Feasibility Problem} (LFP)) and {\bf
  NP}-hard (e.g.~the {\sc Integer Feasibility Problem} (IFP), which
can naturally model {\sc sat}, and also see \cite{NPcomplete}).

In this paper, we use a random linear projection operator $T$ to embed
both $p$ and $X$ to a lower dimensional space, and study the
relationship between the original membership problem and its projected
version:
\begin{quote}
  {\sc Projected ESMP} (PESMP). Given $p,X,T$ as above, decide whether
  $T(p) \in T(X)$.  
\end{quote}
Note that, when $p\in X$ ,the fact that $T(p)\in T(X)$ follows by
linearity of $T$. We are therefore only interested in the case when
$p\notin X$, i.e.~we want to estimate $\mbox{\sf Prob}(T(p) \notin
T(X))$, given that $p \notin X$.

\subsection{Previous results}
Random projections applying to some special cases of membership
problems have been studied in \cite{KPLreport}, where we exploited
some polyhedral structures of the problem to derive several results
for polytopes and polyhedral cones. In the case $X$ is a polytope, we
obtained the following result.
\begin{prop}[\cite{KPLreport}]
  Given $a_1,\ldots,a_n \in \mathbb{R}^{m}$, let $C=\mbox{\sf conv}
  \{a_1,\ldots,a_n\}$, $b \in \mathbb{R}^m$ such that $b \notin C$, $d
  = \min\limits_{x \in C} \|b-x\|$ and $D = \max\limits_{1\le i \le n}
  \|b -a_i\|$. Let $T:\mathbb{R}^m \to \mathbb{R}^k$ be a JL random
  mapping. Then
  \begin{equation*}
    \mbox{\sf Prob} \big(T(b) \notin T(C) \big) \ge 1 - 2n^2
    e^{-\mathcal{C}(\varepsilon^2-\varepsilon^3)k}
  \end{equation*}
  for some constant $\mathcal{C}$ (independent of $m,n,k,d,D$) and
  $\varepsilon < \frac{d^2}{D^2}$.
\end{prop}
If $X$ is a polyhedral cone, we obtained the following result.
\begin{prop}[\cite{KPLreport}]
  Given $b, a_1,\ldots,a_n \in \mathbb{R}^{m}$ of norms $1$ such that
  $b\notin C= \mbox{\sf cone} \{a_1,\ldots,a_n\}$, let $d =
  \min\limits_{x \in C} \|b-x\|$ and $T:\mathbb{R}^m \to \mathbb{R}^k$
  be a JL random mapping. Then:
  \begin{equation*}
    \mbox{\sf Prob}  \big(T(b) \notin T(C) \big) \ge 1 -
    2n(n+1)e^{-\mathcal{C}(\varepsilon^2-\varepsilon^3)k}
  \end{equation*}
  for some constant $\mathcal{C}$ (independent of $m,n,k,d$), where
  $\varepsilon = \frac{d^2}{ \mu_A^2 + 2 \sqrt{1 - d^2}\mu_A + 1}$,
  \begin{equation*}
    \mu_A = \max \{\|x\|_A \;|\; x \in \mbox{\sf
      cone}(a_1,\ldots,a_n) \land \|x\| \le 1\},
  \end{equation*}
  and $\|x\|_A=\min\big\{\sum_i\theta_i\;|\;\theta\ge 0\land
  x=\sum_i\theta_ia_i\big\}$ is the norm induced by
  $A=(a_1,\ldots,a_n)$.
\end{prop}
We also recall the following Lemma, useful for the integer case.
\begin{lem}[\cite{KPLreport}] \label{membershiplemma}
Let $T:\mathbb{R}^m \to \mathbb{R}^k$ be a JL random mapping, let
$b,a_1,\ldots,a_n\in\mathbb{R}^m$ and let $X \subseteq \mathbb{R}^m$
be a finite set. Then if $b \neq \sum_{i=1}^m y_i a_i$ for all $y \in
X$, we have
  \[\mbox{\sf Prob}\, \big(\forall y\in X\;|\; 
    T(b) \neq \sum_{i=1}^m y_i T(a_i) \big) \ge 1 -
  2|X|e^{-\mathcal{C}k};\] \label{lm3}
for some constant $\mathcal{C}>0$ (independent of $m,k$).
\end{lem}

\subsection{New results}
In this paper, we consider the general case where the data set $X$ has
no specific structure, and use Gaussian random projections in our
arguments to obtain some results about the relationship between ESMP
and PESMP.

In the case when $X$ is at most countable (i.e. finite or countable),
using a straightforward argument, we prove that these two problems are
equivalent almost surely. However, this result is only of theoretical
interest due to round-off errors in floating point operations, which
make its practical application difficult. We address this issue by
introducing a threshold $\delta>0$ with a corresponding {\sc Threshold
  ESMP} (TESMP): if $\Delta$ is the distance between $T(p)$ and the
closest point of $T(X)$, decide whether $\Delta\ge\delta$.

In the case when $X$ may also be uncountable, we employ the {\it
  doubling constant} of $X$, i.e.~the smallest number $\lambda_X$ such
that any closed ball in $X$ can be covered by at most $\lambda_X$
closed balls of half the radius. Its logarithm $\log_2 \lambda_X$ is
called {\it doubling dimension} of $X$. Recently, the doubling
dimension has become a powerful tool for several classes of problems
such as nearest neighbor \cite{NavigateSet,NNembed}, low-distortion
embeddings \cite{ahy07}, clustering \cite{magen}.

We show that we can project $X$ into $\mathbb{R}^k$, where $k =
O(\log_2 \lambda_X)$, whilst still ensure the equivalence between ESMP
and PESMP with high probability. We also extend this result to the
threshold case, and obtain a more useful bound for $k$.

\section{Finite and countable sets}
In this section, we assume that $X$ is either finite or countable. Let
$T$ be a JL random mapping from a Gaussian distribution, i.e.~each
entry of $T$ is independently sampled from $\mathcal{N}(0,1)$. It is
well known that, for an arbitrary unit vector $a \in
\mathbb{S}^{m-1}$, the random variable $\|Ta\|^2$ has a Chi-squared
distribution $\chi_k^2$ with $k$ degrees of freedom
(\cite{chi-square-book}). Its corresponding density function is
$\frac{2^{-k/2}}{\Gamma(k/2)} x^{k/2 - 1} e^{k/2}$, where
$\Gamma(\cdot)$ is the gamma function. By \cite{ElementaryJLL}, for
any $0<\delta<1$, taking $z=\frac{\delta}{k}$ yields a cumulative
distribution function
\begin{equation}
\label{gauss-bound2}
  F_{\chi_k^2}(\delta) \le (z e^{1-z})^{k/2} < (ze)^{k/2} =
  \left(\frac{e\delta }{k}\right)^{k/2}. 
\end{equation}
Thus, we have
\begin{equation}
  \label{gauss-bound1}
  \mbox{\sf Prob}(\|Ta\|  \le \delta ) = F_{\chi_k^2}(\delta^2) <
  (3\delta^2)^{k/2} 
\end{equation}
or, more simply, $\mbox{\sf Prob} (\|Ta\| \le \delta ) < \delta^k$
when $k \ge 3$.

Using this estimation, we immediately obtain the following result.
\begin{prop}
Given $p\in\mathbb{R}^m$ and $X\subseteq\mathbb{R}^m$, at most
countable, such that $p \notin X$. Then, for a Gaussian random
projection $T:\mathbb{R}^m\to \mathbb{R}^k$ with any $k \ge 1$, we
have $T(p)\notin T(X)$ almost surely, i.e.~$\mbox{\sf Prob}\big(T(p)
\notin T(X)\big) = 1$.
\label{p:1row}
\end{prop}
\begin{proof}
First, note that for any $u\neq 0$, $Tu \neq 0$ holds almost
certainly. Indeed, without loss of generality we can assume that
$\|u\| = 1$. Then for any $0 < \delta < 1$:
\begin{equation*}
  \mbox{\sf Prob}\big(T(z) = 0\big)\le \mbox{\sf Prob}\big(\|Tz\| \le
  \delta \big) = (3\delta^2)^{k/2} \to 0 \mbox{ as } \delta \to 0.
\end{equation*}
Since the event $T(p)\notin T(X)$ can be written as the intersection
of at most countably many almost sure events $T(p) \neq T(x)$ (for $x
\in X$), it follows that $\mbox{\sf Prob}\big(T(p) \notin T(X)\big) =
1$, as claimed. 
\end{proof}

Proposition \ref{p:1row} is simple, but it looks interesting because
it suggests that we only need to project the data points to a line
(i.e.~$k = 1$) and study an equivalent membership problem on a line.
Furthermore, it turns out that this result remains true for a large
class of random projections.
\begin{prop}
Let $\nu$ be a probability distribution on $\mathbb{R}^m$ with bounded
Lebesgue density $f$. Let $Y\subseteq\mathbb{R}^m$ be an at most
countable set such that $0 \notin Y$. Then, for a random projection
$T:\mathbb{R}^m\to \mathbb{R}^1$ sampled from $\nu$, we have $0 \notin
T(Y)$ almost surely, i.e.~$\mbox{\sf Prob}\big(0 \notin T(Y)\big) =
1$.
\label{pprime:1row}
\end{prop}
\begin{proof}
For any $0 \neq y \in Y$, consider the set
$\mathcal{E}_y=\{T:\mathbb{R}^m\to \mathbb{R}^1 \; |\; T(y)=0 \}$.  If
we regard each $T:\mathbb{R}^m\to \mathbb{R}^1$ as a vector $t \in
\mathbb{R}^m$, then $\mathcal{E}_y$ is a hyperplane $\{t \in
\mathbb{R}^m | \; y\cdot t=0 \}$ and we have
$$\mbox{\sf Prob}(T(y)=0) = \nu(\mathcal{E}_y) = \int_{\mathcal{E}_y}
fd\mu\leq \|f\|_{\infty}\int_{\mathcal{E}_y} d\mu = 0$$ where $\mu$
denotes the Lebesgue measure on $\mathbb{R}^m$. The proof then follows
by the countability of $Y$, similarly to Proposition \ref{p:1row}.
\end{proof}

Proposition \ref{pprime:1row} is based on the observation that the
	degree $[\mathbb{R}:\mathbb{Q}]$ of the field extension
	$\mathbb{R}/\mathbb{Q}$ is $2^{\aleph_0}$, whereas $Y$ is countable;
	so the probability that any row vector $T_i$ of the random
	projection matrix $T$ will yield a linear dependence relation
	$\sum_{j\le m} T_{ij}y_j=0$ for some $0\not=y\in Y$ is zero. In
	practice, however, $Y$ is part of the rational input of a decision
	problem, and the components of $T$ are rational: hence any
	subsequence of them is trivially linearly dependent over
	$\mathbb{Q}$. Moreover, floating point numbers have a bounded binary
	representation: hence, even if $Y$ is finite, there is a nonzero
	probability that any subsequence of components of $T$ will be
	linearly dependent by means of a nonzero multiplier vector in $Y$.

This idea, however, does not work in practice: we tested it by
considering the ESMP given by the IPF defined on the set
$\{x\in\mathbb{Z}^n_+\cap[L,U]\;|\;Ax=b\}$. Numerical experiments
indicate that the corresponding PESMP
$\{x\in\mathbb{Z}^n_+\cap[L,U]\;|\;T(A)x=T(b)\}$, with $T$ consisting
of a one-row Gaussian projection matrix, is always feasible despite
the infeasibility of the original IPF. Since Prop.~\ref{p:1row}
assumes that the components of $T$ are real numbers, we think that the
reason behind this failure is the round-off error associated to the
floating point representation used in computers. Specifically, when
$T(A)x$ is too close to $T(b)$, floating point operations will
consider them as a single point. In order to address this issue, we
force the projected problems to obey stricter requirements. In
particular, instead of only requiring that $T(p) \notin T(X)$, we
ensure that
\begin{equation*}
  \mbox{\sf dist}(T(p),T(X)) = \min_{x \in X} \; \|T(p) - T(x)\|>\tau,
\end{equation*}
where $\mbox{\sf dist}$ denotes the Euclidean distance, and $\tau>0$
is a (small) given constant. With this restriction, we obtain the
following result.
\begin{prop}
  Given $\tau,\delta>0$  and $p\notin
  X\subseteq\mathbb{R}^m$, where $X$ is a finite set, let
  \begin{equation*}
    d = \min\limits_{x \in X} \; \|p - x\| > 0.
  \end{equation*}
  Let $T:\mathbb{R}^m\to\mathbb{R}^k$ be a Gaussian random projection
  with $k\ge \frac{\log(\frac{|X|}{\delta})}{\log(\frac{d}{\tau})}$. Then:
  \begin{equation*}
    \mbox{\sf Prob}\big(\min_{x \in X} \; \|T(p) - T(x)\| > \tau \big)
    > 1 - \delta.
  \end{equation*}
\end{prop}
\begin{proof}
We assume that $k\ge 3$. For any $x\in X$ we have:
\begin{eqnarray*}
  \mbox{\sf Prob}\big(\|T(p-x)\| \le \tau\big) &=& \mbox{\sf
    Prob}\bigg(\bigg\|T\big(\frac{p-x}{\|p-x\|}\big)\bigg\| \le
  \frac{\tau}{\|p-x\|} \bigg) \\ & \le& \mbox{\sf Prob}\bigg(\bigg
  \|T\big(\frac{p-x}{\|p-x\|}\big)\bigg\| \le \frac{\tau}{d} \bigg) <
  \frac{\tau^k }{d^k},
\end{eqnarray*}
due to (\ref{gauss-bound2}). Therefore, by the union bound,
\begin{align*}
\mbox{\sf Prob}\big(\min\limits_{x \in X} \; \|T(p) - T(x)\| > \tau
\big) & = 1 - \mbox{\sf Prob}\big(\min\limits_{x \in X} \; \|T(p) -
  T(x)\| \le \tau \big) \\ 
  & \ge 1 - \sum\limits_{x \in X} \mbox{\sf
    Prob}\big(\|T(p) - T(x)\| \le \tau \big) > 1 - |X|  \left(\frac{\tau}{d}\right)^k. 
\end{align*}
The RHS is greater than or equal to $1 - \delta$ if and only if
$\left(\frac{d}{\tau}\right)^k\ge \frac{|X|}{\delta}$, which is
equivalent to $k\ge\frac{\log (\frac{|X|}{\delta})}{\log
  (\frac{d}{\tau})}$, as claimed.
\end{proof}
Note that $d$ is often unknown and can be arbitrarily small. However,
if both $p,X$ are integral, then $d \ge 1$ and we can select
$k>\frac{\log \frac{|X|}{\delta}}{\log \frac{1}{\tau}}$ in the above
proposition. 

In many cases, the set $X$ is infinite. We show that when this is the
case, we can still overcome this difficulty under some assumptions. In
particular, we prove that if $X=\{Ax\;|\;x\in \mathbb{Z}^n_+\}$ where
$A$ is an $m\times n$ matrix with integer coefficients which are all
positive in at least one row, then for any bounded vector $b \in
\mathbb{Z}^m$ the problem $b \in X$ is equivalent, with high
probability, to its projection to a $O(\log n)$-dimensional space. The
idea is to separate one positive row and apply random projection to
the others.

Formally, let us denote by $a^i$ the $i$-th row and by $a_j$ the
$j$-th column of $A$. Assume that all entries in the row $a^i$ is
positive and all entries of $b$ are bounded by a constant $B >
0$. Remove the row $i$ from $A$ and $b$ to obtain
$\tilde{A}=(a'_1,\ldots,a'_n) \in \mathbb{Z}^{(m-1)\times n}$ and
$\tilde{b}\in\mathbb{Z}^{m-1}$.  Let $T: \mathbb{R}^{m-1} \to
\mathbb{R}^k$ be a JL random mapping and denote by
$Z=\{x\in\mathbb{Z}^n_+\;|\;a^i\cdot x=b_i\}$. Then we have: 
\begin{prop}
 Assume that $b\notin X$, and let $0<\delta<1$. Using the terminology
 and given the assumptions above, if
 $k\ge\frac{1}{\mathcal{C}}\ln(\frac{2}{\delta})+\frac{B}{\mathcal{C}}
  \log(n+B-1)$ we have
 \[
 \mbox{\sf Prob}\bigg(T(b)\neq
 \sum\limits_{j=1}^n x_jT(a'_j)  \mbox{ for all }x\in Z\bigg)\geq
 1-\delta
 \]
 for some constant $\mathcal{C}>0$.
 \label{prop:PL}
\end{prop}
\begin{proof}
  We first show that $|Z| \leq (n+B-1)^{B}$. Since all the entries of
  $A$ are positive integers, we have
  \[|Z| \leq |\{x \in \mathbb{Z}^n_+ \; |\;
    \sum\limits_{j=1}^nx_j=b_i\}|\leq |\{x \in \mathbb{Z}^n_+ \; |\;
    \sum\limits_{j=1}^nx_j=B\}|.
    \]
  The number of elements in the RHS corresponds to the
  number of combinations with repetitions of $B$ items sampled from
  $n$, which is equal to
  ${n+B-1 \choose n-1} = {n+B-1 \choose B} \leq (n+B-1)^{B}$.

  Next, by Lemma \ref{membershiplemma}, we have:
  \begin{equation}
    \mbox{\sf Prob}\bigg(T(b)\neq \sum\limits_{j=1}^n x_jT(a'_j)
    \mbox{ for all }x\in Z\bigg) \ge 1 -
    2(n+B-1)^{B}e^{-\mathcal{C}k}, \label{eq:rlm}
  \end{equation}
  which is greater than $1 - \delta$ when taking any $k$ such that $k
    \ge \frac{1}{\mathcal{C}}\ln(\frac{2}{\delta}) +
    \frac{B}{\mathcal{C}} \log(n+B-1)$. The proposition is proved.
\end{proof}
Note that in Prop.~\ref{prop:PL} we can choose the JL random mapping
$T$ as a matrix with $\{-1, +1\}$ entries (Rademacher variables). In
this case, there is no need to worry about floating point errors.

\section{Sets with low doubling dimension}
In this section, we denote by $B(x,r)$ the closed ball centered at $x$
with radius $r > 0$, and $B_X(x,r)=B(x,r)\cap X$. We will also assume
that $X$ is a doubling space, i.e. a set with bounded doubling
dimension. One example of doubling spaces is a Euclidean space.
$\mathbb{R}^m$, we can show that the doubling dimension
$\log_2(\lambda_X)$ of $X$ can be shown to be a constant factor of $m$
(\cite{ddeuclid, fractal}). However, many sets of low doubling
dimensions are contained in high dimensional spaces
(\cite{intrinsic}). Note that computing the doubling dimension of a
metric space is generally {\bf NP}-hard (\cite{ProxNear}). We shall
make use of the following simple lemma.
\begin{lem}
  \label{keylem}
  For any $p\in X$ and $\varepsilon,r>0$, there is a set $S\subseteq
  X$ of size at most $\lambda_X^{\lceil
    \log_2(\frac{r}{\varepsilon})\rceil}$ such that
  \begin{equation*}
    B_X(p,r)\subseteq \bigcup_{s \in S_j} B(s, \varepsilon).
  \end{equation*}
\end{lem}
\begin{proof}
  By definition of the doubling dimension, $B_X(p,r)$ is covered by at
  most $\lambda_X$ closed balls of radius $\frac{r}{2}$. Each of these
  balls in turn is covered by $\lambda_X$ balls of radius
  $\frac{r}{4}$, and so on: iteratively, for each $k \ge 1$,
  $B_X(p,r)$ is covered by $\lambda_X^k$ balls of radius
  $\frac{r}{2^k}$. If we select $k=\lceil
  \log_2(\frac{r}{\varepsilon})\rceil$ then $k \ge \log_2
  (\frac{r}{\varepsilon})$, i.e. $\frac{r}{2^k} \le \varepsilon$. This
  means $B_X(p,r)$ is covered by $\lambda_X^{\lceil
    \log_2(\frac{r}{\varepsilon})\rceil}$ balls of radius
  $\varepsilon$.
\end{proof}

We will also use the following lemma, which is proved in
\cite{NNembed} using a concentration estimation for sum of squared
gaussian variables (Chi-squared distribution).
\begin{lem}
  \label{lem1}
  Let $X \subseteq B(0,1)$ be a subset of the $m$-dimensional
  Euclidean unit ball. Then there exist universal constants $c,C>0$
  such that for $k \ge C \log \lambda_X +1$ and $\delta>1$, the
  following holds:
  \begin{equation*}
    \mbox{\sf Prob} (\exists x \in X \mbox{ s.t. } \|Tx\|>\delta)<e^{-ck\delta^2}.
  \end{equation*}
\end{lem}
In the proof of the next result (one of the main results in this
section), we use the same idea as that in \cite{NNembed} for the
nearest neighbor problem.
\begin{thm}
  Given $0 < \delta < 1$ and $p \notin X \subseteq \mathbb{R}^m$. Let
  $T: \mathbb{R}^m \to \mathbb{R}^k$ be a Gaussian random
  projection. Then
  \begin{equation*}
    \mbox{\sf Prob}(T(p) \notin T(X)) = 1
  \end{equation*}
  if $k \ge \mathcal{C}\log_2 (\lambda_X)$, for some universal
  constant $\mathcal{C}$.
  \label{t:main}
\end{thm}
\begin{proof}
  Let $\varepsilon>0$ and $0 = r_0 < r_1 < r_2 < \ldots$ be positive
  scalars (their values will be defined later). For each $j =
  1,2,3,\ldots$ we define a set
  \begin{equation*}
    X_j = X \cap B(p, r_j) \smallsetminus B(p, r_{j-1}).
  \end{equation*}

  Since $X_j\subseteq B_X(p,r_j)$, by Lemma \ref{keylem} we can find a
  point set $S_j\subseteq X$ of size $|S_j|\le\lambda_X^{\lceil
    \log_2(\frac{r_j}{\varepsilon})\rceil}$ such that
  \begin{equation*}
    X_j \subseteq \bigcup_{s \in S_j} B(s, \varepsilon).
  \end{equation*}
  Hence, for any $x \in X_j$, there is $s \in S_j$ such that
  $\|x-s\|<\varepsilon$. Moreover, by the triangle inequality, any
  such $s$ satisfies $r_{j-1}-\varepsilon<\|s - p\|<r_j+\varepsilon$,
  so without loss of generality we can assume that
  \begin{equation*}
    S_j\subseteq B(p,r_j+\varepsilon)\smallsetminus B(p,
    r_{j-1}-\varepsilon).
  \end{equation*}
  We denote by $\mathcal{E}_j$ the event that:
  \begin{equation*}
    \exists s \in S_j, \; \exists x \in X_j \cap B(s, \varepsilon)
    \mbox{ s.t. } \|Ts-Tx\| > \varepsilon \sqrt{j}.
  \end{equation*}
  By the union bound, we have
  \begin{eqnarray*}
    \mbox{\sf Prob}(\mathcal{E}_j) & \le & \sum_{s \in S_j} \mbox{\sf
      Prob}\big(\exists x \in X_j  \cap B(s, \varepsilon) \mbox{ s.t. } \|Ts - Tx\| > \varepsilon \sqrt{j}\big) \\ 
    & \le &  \sum_{s \in S_j} e^{-c_1kj} \qquad \mbox{ (for some
      universal constant $c_1$ by Lemma \ref{lem1})} \\ 
    &\le& \lambda_X^{\lceil \log_2(\frac{r_j +
        \varepsilon}{\varepsilon})\rceil} \; e^{-c_1kj}. 
  \end{eqnarray*}
  Again by the union bound, we have:
  \begin{eqnarray*}
    \mbox{\sf Prob} \big(\exists x \in X \mbox{ s.t } T(x) = T(p)\big)
    & =& \mbox{\sf Prob}\big(\exists x \in \bigcup_{j = 1}^\infty X_j
    \mbox{ s.t } T(x) = T(p)\big) \\ 
    & \le& \sum_{j=1}^\infty \mbox{\sf Prob}\big(\exists x \in X_j
    \mbox{ s.t } T(x) = T(p)\big).  
  \end{eqnarray*}
  Now we will estimate the individual probabilities:
  \begin{eqnarray*}
    & \;& \mbox{\sf Prob} \big(\exists x \in X_j \mbox{ s.t } T(x) =
    T(p) \big) \\ 
    &\le &  \mbox{\sf Prob} \big((\exists x \in X_j \mbox{ s.t } T(x)
    = T(p)) \wedge \mathcal{E}^c_j\big) + \mbox{\sf Prob} (\mathcal{E}_j) \\
    & \le & \mbox{\sf Prob}\big(\exists x \in X_j, s \in S_j
    \cap B(x, \varepsilon) \mbox{ s.t } T(x) = T(p)\land \|T(s) - T(x)\|
    \le \varepsilon \sqrt{j}\big) + \mbox{\sf Prob} (\mathcal{E}_j)  \\ 
    & \le & \mbox{\sf Prob}\big(\exists s \in S_j \mbox{ s.t
    }\|T(s)-T(p)\|<\varepsilon\sqrt{j}\big)  +  \lambda_X^{\lceil  
  \log_2(\frac{r_j + \varepsilon}{\varepsilon})\rceil}\,e^{-c_1kj}. 
  \end{eqnarray*}
  Next, we choose $\varepsilon = \frac{d}{N}$ for some large $N$; and
  for each $j \ge 1$, we choose $r_j = (2 + j) \varepsilon$.  For $j <
  N - 2$, by definition it follows that $X_j = \varnothing$. Therefore
  \begin{equation*}
    \mbox{\sf Prob}\big(\exists x \in X_j \mbox{ s.t }  T(s) = T(p)\big)  = 0.
  \end{equation*} 
  On the other hand, for $j \ge N-2$,
  \begin{eqnarray*}
	& & \mbox{\sf Prob}\big(\exists s \in S_j \mbox{ s.t }  \|T(s)
    - T(p)\| \le \varepsilon \sqrt{j}\big) \\ 
    & \le & \lambda_X^{\lceil \log_2(\frac{r_j +
        \varepsilon}{\varepsilon})\rceil} \;  \mbox{\sf Prob}\big(
      \|T(z)\| \le \frac{\varepsilon \sqrt{j}}{r_{j-1} -
        \varepsilon}\big)  \quad \mbox{ for an arbitrary } z \in
    \mathbb{S}^{n-1} \\ 
    & = & \lambda_X^{\lceil \log_2(3+j)\rceil}\, \mbox{\sf Prob}\big(
    \|T(z)\| \le \frac{1} {\sqrt{j}}\big) \quad \mbox{ for an
  arbitrary } z \in \mathbb{S}^{n-1} \\ 
    & < & \lambda_X^{\lceil \log_2(3+j)\rceil}\, j^{-k/2} \qquad
    \mbox{ by the estimation \eqref{gauss-bound1}}. 
  \end{eqnarray*}
  Note that $\lambda_X^{\lceil \log_2(3+j)\rceil} \le
  \lambda_X^{\log_2(6+2j)} = (6 + 2j)^{\log_2 \lambda_X} < j^{(2
    \log_2 \lambda_X)}$ for large enough $N$. Therefore, we have
  \begin{eqnarray*}
    \mbox{\sf Prob} \big(\exists x \in X_j \mbox{ s.t } T(x) = T(p)
    \big)  & \le&  \lambda_X^{\lceil \log_2(3+j)\rceil}\,
\big(j^{-k/2} +  e^{-c_1kj} \big) \\ & \le & j^{-c_2k} +  e^{-c_3kj}  
  \end{eqnarray*}
  for some universal constants $c_2, c_3$, provided that $k \ge
  \mathcal{C}_1 \log \lambda_X$ for some large enough constant
  $\mathcal{C}_1$. Finally, by the union bound,
  \begin{eqnarray*}
    \mbox{\sf Prob} \big(T(p) \notin T(X) \big) & =& 1 -  \mbox{\sf
      Prob} \big(T(p) \in T(X) \big) \\ 
    & \ge&  1 - \sum\limits_{i=N-2}^\infty \big(i^{-c_2k} +
      e^{-c_3kj} \big)
  \end{eqnarray*}
  which tends to $1$ when $N$ tends to infinity. 
\end{proof}

Our final result in the section is an extension of Thm.~\ref{t:main}
to the threshold case.
\begin{thm}
  Let $p \notin X \subseteq \mathbb{R}^m$, $T:\mathbb{R}^m \to
  \mathbb{R}^k$ be a Gaussian random projection, and $d=\min\limits_{x
    \in X} \|p - x\|$.  Then for all $0 < \delta < 1$ and all
  $0 <\tau < \kappa d$ for some constant $\kappa < 1$, we have
  \begin{equation*}
    \mbox{\sf Prob}(\mbox{\sf dist}(T(p),T(X)) > \tau) >  1 - \delta
  \end{equation*}
  if $k$ is $O(\frac{\log (\frac{\lambda_X}{\delta} ) }{\log
  	(\frac{d}{\tau})})$.
\end{thm}
\begin{proof}
  For $j = 1,2,\ldots $ we construct the sets $X_j, S_j$ similarly as
  those in the proof of Thm.~\ref{t:main} (where the values of $r_j$
  and $\varepsilon$ will be defined later). Then we have
  \begin{eqnarray*}
    \mbox{\sf Prob} \big(\exists x \in X \mbox{ s.t } \|T(x) -
    T(p)\| < \tau \big)   
    & =& \mbox{\sf Prob}\big(\exists x \in \bigcup_{j = 1}^\infty X_j
    \mbox{ s.t }  \|T(x) - T(p)\| < \tau\big) \\ 
    & \le& \sum\limits_{j=1}^\infty \mbox{\sf Prob}\big(\exists x \in
    X_j \mbox{ s.t }  \|T(x) - T(p)\| < \tau \big).  
  \end{eqnarray*}
  For all $j \ge 1$, we have
  \begin{eqnarray*}
    && \mbox{\sf Prob} \big(\exists x \in X_j \mbox{ s.t } \|T(x) - T(p)\| < \tau\big) \\
    & \le & \mbox{\sf Prob} \big((\exists x \in X_j \mbox{ s.t } \|T(x)
    - T(p)\| < \tau) \wedge \mathcal{E}^c_j\big) + \mbox{\sf
      Prob}(\mathcal{E}_j) \\  
    & \le & \mbox{\sf Prob}\big(\exists x \in X_j, s \in S_j
    \cap B(x, \varepsilon) \mbox{ s.t } \|T(x) - T(p)\| < \tau\land \|T(s)
    - T(x)\| \le \varepsilon \sqrt{j}\big) + \mbox{\sf Prob}
    (\mathcal{E}_j) \\
    & \le & \mbox{\sf Prob}\big(\exists s \in S_j \mbox{ s.t }  \|T(s)
    - T(p)\| < \tau + \varepsilon \sqrt{j}\big)  +  \lambda_X^{\lceil
        \log_2(\frac{r_j + \varepsilon}{\varepsilon})\rceil} \;
      e^{-c_1kj}. 
  \end{eqnarray*}
  Now we choose $\varepsilon = \frac{\tau}{N}$ for some $N > 0$ such that $1 + \frac{1}{N} < \frac{1}{\kappa}$ and for each $j \ge 1$, we choose
  $r_j = \tau \sqrt{j+1} + (2 + j) \varepsilon$.  
  For $j = 1$, by the union bound we have
  \begin{eqnarray}
  	&& \mbox{\sf Prob}\big(\exists s \in S_1 \mbox{ s.t }  \|T(s) -
  	T(p)\| \le \tau + \varepsilon \sqrt{1}\big) \nonumber \\ [-0.2em]
  	& \le & \lambda_X^{\lceil \log_2(\frac{r_1 +
  			\varepsilon}{\varepsilon})\rceil}\,\mbox{\sf
  		Prob}\big(\|T(z)\| \le \frac{\tau + \varepsilon}{d}\big) \quad \mbox{ for an
  		arbitrary } z \in \mathbb{S}^{m-1} \nonumber \\[-0.2em]
  	& = & \lambda_X^{\lceil \log_2(4 + N\sqrt{2})\rceil}\,\mbox{\sf
  		Prob}\bigg(\|T(z)\| \le (1 +\frac{1} {N})\frac{\tau}{d}\bigg)  \quad \mbox{
  		for an arbitrary } z \in \mathbb{S}^{m-1} \nonumber \\[-0.2em] 
  	& < & \lambda_X^{\lceil \log_2(4+ N\sqrt{2})\rceil}\, \bigg((1 +\frac{1} {N}) \frac{\tau}{d}\bigg)^{k/2}
  	\qquad \mbox{ by estimation \eqref{gauss-bound1}} \nonumber \\[-0.2em]
  	& < & \bigg((1 +\frac{1} {N}) \frac{\tau}{d}\bigg)^{c_2k}		\label{concl1}
  \end{eqnarray}
  for some universal constant $c_2 > 0$, as long as $k > \mathcal{C} \log (\lambda_X)$ for some $\mathcal{C}$ large enough. 
  
 For $j \ge 2$, we
 have
 \begin{eqnarray}
 	&& \mbox{\sf Prob}\big(\exists s \in S_j \mbox{ s.t }  \|T(s) -
 	T(p)\| \le \tau + \varepsilon \sqrt{j}\big) \nonumber \\ [-0.2em]
 	& \le & \lambda_X^{\lceil \log_2(\frac{r_j +
 			\varepsilon}{\varepsilon})\rceil}\,\mbox{\sf
 		Prob}\big(\|T(z)\| \le \frac{\tau + \varepsilon
 		\sqrt{j}}{r_{j-1} - \varepsilon}\big) \quad \mbox{ for an
 		arbitrary } z \in \mathbb{S}^{m-1} \nonumber \\ [-0.2em]
 	& = & \lambda_X^{\lceil \log_2(3+j + N\sqrt{j+1})\rceil}\,\mbox{\sf
 		Prob}\big(\|T(z)\| \le \frac{1} {\sqrt{j}}\big)  \quad \mbox{
 		for an arbitrary } z \in \mathbb{S}^{m-1} \nonumber \\ [-0.2em]
 	& < & \lambda_X^{\lceil \log_2(3+j + N\sqrt{j+1})\rceil}\, j^{-k/2}
 	\qquad \mbox{ by estimation \eqref{gauss-bound1}} \nonumber \\[-0.2em]
 	& < & j^{-c_3 k} \label{concl2}
 \end{eqnarray}
 for some universal constant $c_3 > 0$, as long as $k > \mathcal{C} \log (\lambda_X)$ for some $\mathcal{C}$ large enough.
 
Similarly, for all $1 \le j$, we have
\begin{eqnarray}
\lambda_X^{\lceil
	\log_2(\frac{r_j + \varepsilon}{\varepsilon})\rceil} \;
e^{-c_1kj} \le  e^{-c_4kj},	\label{concl3}
\end{eqnarray}
for some universal constant $c_4 > 0$, 
as long as $k > \mathcal{C} \log (\lambda_X)$ for some $\mathcal{C}$ large enough.

   From estimations (\ref{concl1}), (\ref{concl2}), (\ref{concl3}) and by the union bound we have:
    \begin{eqnarray*}
  	\mbox{\sf Prob}(\mbox{\sf dist}(T(p),T(X)) \ge \tau) 
  	& \ge& 1 - \sum_{j=1}^\infty \mbox{\sf Prob}(\mbox{\sf dist}(T(p),T(X_j)) < \tau) \\[-0.2em]
  	& \ge & 1 -  \bigg((1 +\frac{1} {N}) \frac{\tau}{d}\bigg)^{c_2k}	 - \sum_{j=2}^\infty j^{-c_3k}- \sum_{j=1}^\infty  e^{-c_4kj}\\ [-0.2em]
  	& \ge& 1 - \delta \qquad \mbox{ for $k = O(\frac{\log (\frac{\lambda_X}{\delta} ) }{\log
  			(\frac{d}{\tau})})$ large enough.}
  \end{eqnarray*}
  \end{proof}


\end{document}